\documentclass[graybox]{svmult}

\usepackage{mathptmx}       
\usepackage{helvet}         
\usepackage{courier}        
\usepackage{type1cm}        
\usepackage{makeidx}         
\usepackage{graphicx}        
\usepackage{multicol}        
\usepackage[bottom]{footmisc}

\makeindex             

\newtheorem{defin}{Definition}
\newtheorem{rem}{Remark}
\newtheorem{rems}{Remarks}
\newtheorem{prop}{Proposition}
\newtheorem{corl}{Corollary}

\usepackage{amsmath}
\usepackage{amssymb}
\usepackage{graphicx}

\def\Z{\mathbb Z}
\def\I{\mathbb I}

\newcommand{\N}{\mathbb N}
\newcommand{\R}{\mathbb R}

\begin{document}

\title*{Discrete approximation to solution flows of Tanaka's SDE related to Walsh Brownian motion}
\titlerunning{Approximation to flows of Tanaka's SDE}
\author{Hatem Hajri}
\institute{Hatem Hajri \at Universit\'e Paris Sud 11, \email{Hatem.Hajri@math.u-psud.fr}}
%
%
\maketitle

\abstract*{In a previous work, we have defined a Tanaka's SDE related to Walsh Brownian motion which depends on kernels. It was shown that there are only one Wiener solution and only one flow of mappings solving this equation. In the terminology of Le Jan and Raimond, these are respectively the stronger and the weaker among all solutions. In this paper, we obtain these solutions as limits of  discrete models.}

\abstract{In a previous work, we have defined a Tanaka's SDE related to Walsh Brownian motion which depends on kernels. It was shown that there are only one Wiener solution and only one flow of mappings solving this equation. In the terminology of Le Jan and Raimond, these are respectively the stronger and the weaker among all solutions. In this paper, we obtain these solutions as limits of  discrete models.}

\section{Introduction and main results}
Consider Tanaka's equation:
\begin{equation}\label{ferjeni}
\varphi_{s,t}(x)=x+\int_{s}^t \textrm{sgn}(\varphi_{s,u}(x))dW_u,\ \ s\leq t, x\in\R,
\end{equation}
where ${\textrm{sgn}}(x)=1_{\{x>0\}}-1_{\{x\leq0\}}, W_t=W_{0,t}1_{\{t>0\}}-W_{t,0} 1_{\{t\leq0\}}$ and $(W_{s,t}, s\leq t)$ is a real white noise on a probability space $(\Omega,\mathcal A,\mathbb P)$ (see Definition 1.10 \cite{MR2060298}). This is an example of a stochastic differential equation which admits a weak solution but has no strong solution. If $K$ is a stochastic flow of kernels (see Section 2.1 \cite{MR5010101}) and $W$ is a real white noise, then by definition, $(K,W)$ is a solution of Tanaka's SDE if for all $s\leq t, x\in\R$, $f\in C^2_b(\R)$ ($f$ is $C^2$ on $\R$ and $f',f''$ are bounded), 
\begin{equation}\label{dom}
K_{s,t}f(x)=f(x)+\int_s^tK_{s,u}(f'{\textrm{sgn}})(x)W(du)+\frac{1}{2}\int_s^t K_{s,u}f''(x)du\ \ a.s.
\end{equation}
When $K=\delta_{\varphi}$ is a flow of mappings, $K$ solves (\ref{dom}) if and only if $\varphi$ solves (\ref{ferjeni}) by It\^o's formula. In \cite{MR2235172}, Le Jan and Raimond have constructed the unique flow of mappings associated to (\ref{ferjeni}). It was shown also that 
\[
K^{W}_{s,t}(x)=\delta_{x+\textrm{sgn}(x)W_{s,t}} 1_{\{t\leq \tau_{s,x}\}}+\frac{1}{2}(\delta_{W^+_{s,t}}+\delta_{-W^+_{s,t}}) 1_{\{t>\tau_{s,x}\}},\ \ s\leq t, x\in\R,
\]
is the unique $\mathcal F^W$ adapted solution (Wiener flow) of (\ref{dom}) where $$\tau_{s,x}=\inf\{r\geq s: W_{s,r}=-|x|\}, \ W_{s,t}^+:=W_{s,t}-\displaystyle\inf_{u\in[s,t]}W_{s,u}.$$
In \cite{MR5010101}, an extension of (\ref{dom}) in the case of Walsh Brownian motion was defined as follows
\begin{defin}
Fix $N\in{\N}^{*}$, $\alpha_1,\cdots,\alpha_N > 0$ such that $\displaystyle{\sum_ {i=1}^{N}\alpha_i=1}$ and consider the graph $G$ consisting of $N$ half lines $(D_i)_{1\leq i\leq N}$ emanating from $0$ (see Figure 1). 

\begin{figure}[h]
\begin{center}
\resizebox{8cm}{7cm}{\input{fig_tanaka1.pstex_t}}
\caption{Graph $G$.}
\end{center}
\end{figure}

Let $\vec{e}_i$ be a vector of modulus $1$ such that $D_i=\{h\vec{e}_i,h\geqslant 0\}$ and define for all $z\in G$,  $\vec{e}(z)=\vec{e}_i$ if $z\in D_i, z\neq0$ (convention $\vec{e}(0)=\vec{e}_N$). Define the following distance on $G$:
\\$$\begin{array}{ll}
d(h\vec{e}_i,h'\vec{e}_j)=\begin{cases}
h+h'&\text{if}\  i\neq j, (h,h')\in \R_+^2,\\
|h-h'|\ &\text {if}\ i=j, (h,h')\in \R_+^2.\\
\end{cases}
\end{array}$$
For $x\in G$, we will use the simplified notation $|x|:=d(x,0)$.\\
We equip $G$ with its Borel $\sigma$-field $\mathcal{B}(G)$ and set $G^*=G\setminus\{0\}$. Let $C_b^2(G^*)$ be the space of all $f:G\longrightarrow\R$ such that $f$ is continuous on $G$ and has bounded first and second derivatives ($f'$ and $f''$) on $G^*$ (here $f'(z)$ is the derivative of $f$ at $z$ in the direction $\vec{e}(z)$ for all $z\neq0$), both $\lim_{z\rightarrow0, z\in D_i, z\neq 0}f'(z)$ and $\lim_{z\rightarrow0, z\in D_i, z\neq 0}f''(z)$ exist for all $i\in [1,N]$. Define
$$D(\alpha_1,\cdots,\alpha_N)=\left\{f\in C_b^2(G^*): \displaystyle\sum_{i=1}^{N}\alpha_i \lim_{z\rightarrow0, z\in D_i, z\neq 0}f'(z)=0\right\}.$$
 Now, Tanaka's SDE on $G$ extended to kernels is the following (see Remarks 3 (1) in \cite{MR5010101} for a discussion of its origin).\\
\textbf{Tanaka's equation on $G$.}
On a probability space $(\Omega,\mathcal A,\mathbb P)$, let $W$ be a real white noise and $K$ be a stochastic flow of kernels on $G$. We say that $(K,W)$ solves $(T)$ if for all $s\leq t, f\in D(\alpha_1,\cdots,\alpha_N), x\in G$,
$$K_{s,t}f(x)=f(x)+\int_s^tK_{s,u}f'(x)W(du) + \frac{1}{2}\int_s^tK_{s,u}f''(x)du\ \ a.s.$$
If $K=\delta_\varphi$ is a solution of $(T)$, we just say that $(\varphi,W)$ solves $(T)$.
\end{defin}
\noindent Equation $(T)$ is a particular case of an equation $(E)$ studied in \cite{MR5010101} (it corresponds to $\varepsilon=1$ with the notations of \cite{MR5010101}). It was shown (see Corollary 2 \cite{MR5010101}) that if $(K,W)$ solves $(T)$, then $\sigma(W)\subset\sigma(K)$ and therefore one can just say that $K$ solves $(T)$. We also recall
\begin{theorem} \cite{MR5010101}\label{3li}
 There exists a unique Wiener flow $K^W$ (resp. flow of mappings $\varphi$) which solves $(T)$.
\end{theorem}
\noindent As described in Theorem 1 \cite{MR5010101}, the unique Wiener solution of $(T)$ is simply
\begin{equation}\label{loop}
K_{s,t}^{W}(x)=\delta_{x+\vec{e}(x)W_{s,t}}1_{\{t\leq \tau_{s,x}\}}+\sum_{i=1}^{N}\alpha_i\delta_{\vec{e}_iW^+_{s,t}}1_{\{t> \tau_{s,x}\}}.
\end{equation}
where
\begin{equation}\label{didi}
\tau_{s,x}=\inf\{r\geq s: x+\vec{e}(x)W_{s,r}=0\}=\inf\{r\geq s: W_{s,r}=-|x|\}.
\end{equation}
\noindent However, the construction of the unique flow of mappings $\varphi$ associated to $(T)$ relies on flipping Brownian excursions and is more complicated.  Another construction of $\varphi$ using Kolmogorov extension theorem can be derived from Section 4.1 \cite{MR2235172} similarly to Tanaka's equation. Here, we restrict our attention to discrete models.\\
The one point motion associated to any solution of $(T)$ is the Walsh Brownian motion $W(\alpha_1,\cdots,\alpha_N)$ on $G$ (see Proposition 3 \cite{MR5010101}) which we define as a strong Markov process with c\`adl\`ag paths, state space $G$ and Feller semigroup $(P_t)_{t\geq0}$ as given in Section 2.2 \cite{MR5010101}. When $N=2$, it corresponds to the famous skew Brownian motion \cite{MR606993}.\\
Our first result is the following Donsker approximation of $W(\alpha_1,\cdots,\alpha_N)$ 
\begin{prop}\label{enri}
 Let $M = (M_n)_{n\geq 0}$ be a Markov chain on $G$ started at $0$ with stochastic matrix $Q$ given by:
\begin{equation}\label{sirine}
Q(0,\vec{e}_{i})=\alpha_i,\ \ Q(n\vec{e}_{i},(n+1)\vec{e}_{i})=Q(n\vec{e}_{i},(n-1)\vec{e}_{i})=\frac{1}{2}\ \ \forall i\in [1,N], n\in {\N^{*}}.
\end{equation}
Let $t\longmapsto M(t)$ be the linear interpolation of $(M_n)_{n\geq 0}$ and $M^n_t=\frac{1}{\sqrt{n}}M(nt), n\geq 1$. Then
$$(M^n_t)_{t\geq0}\xrightarrow[\text{$n\rightarrow +\infty $}]{\text{law}}(Z_t)_{t\geq0}$$ in $C([0,+\infty[,G)$ where $Z$ is an $W(\alpha_1,\cdots,\alpha_N)$ started at $0$.
\end{prop}
This result extends that of \cite{MR1838739} who treated the case $\alpha_1=\cdots=\alpha_N=\frac{1}{N}$ and of course the Donsker theorem for the skew Brownian motion (see \cite{MR1975425} for example). We show in fact that Proposition \ref{enri} can be deduced immediately from the case $N=2$.\\
In this paper we study the approximation of flows associated to $(T)$. Among recent papers on the approximation of flows, let us mention \cite{6} where the author construct an approximation for the Harris flow and the Arratia flow.\\
Let $G_{\mathbb N}=\{x\in G; |x|\in \mathbb N\}$ and $\mathcal P(G)$ (resp. $\mathcal P(G_{\mathbb N})$) be the space of all probability measures on $G$ (resp. $G_{\mathbb N}$). We now come to the discrete description of $(\varphi, K^W)$ and introduce
\begin{defin} (Discrete flows)
We say that a process $\psi_{p,q}(x)$ (resp. $N_{p,q}(x)$) indexed by $\{p\leq q\in\Z, x\in G_{\mathbb N}\}$ with values in $G_{\mathbb N}$ (resp. $\mathcal P(G_{\mathbb N})$) is a discrete flow of mappings (resp. kernels) on $G_{\mathbb N}$ if:\\
(i) The family $\{\psi_{i,i+1}; i\in\Z\}$ (resp. $\{N_{i,i+1}; i\in\Z\}$) is independent.\\
$(ii) \forall p\in\Z, x\in G_{\mathbb N}, \psi_{p,p+2}(x)=\psi_{p+1,p+2}(\psi_{p,p+1}(x))\\ (\textrm{resp.}\ N_{p,p+2}(x)=N_{p,p+1}N_{p+1,p+2}(x))\ \textrm{a.s.}$
where 
$$N_{p,p+1}N_{p+1,p+2}(x,A):=\sum_{y\in G_{\mathbb N}}N_{p+1,p+2}(y,A)N_{p,p+1}(x,\{y\})\ \textrm{for all}\ A\subset G_{\mathbb N}.$$
We call (ii), the cocycle or flow property. 
\end{defin}
The main difficulty in the construction of the flow $\varphi$ associated to (\ref{ferjeni}) \cite{MR2235172} is that it has to keep the consistency of the flow. This problem does not arise in discrete time. Starting from the following two remarks, 
\begin{itemize}
\item[(i)] $\varphi_{s,t}(x)=x+\textrm{sgn}(x)W_{s,t}$ if $s\leq t\leq \tau_{s,x}$,
\item[(ii)] $|\varphi_{s,t}(0)|=W^+_{s,t}$ and $\textrm{sgn}(\varphi_{s,t}(0))$ is independent of $W$ for all $s\leq t$,
\end{itemize}
one can easily expect the discrete analogous of $\varphi$ as follows: consider an original random walk $S$ and a family of signs $(\eta_i)$ which are independent. Then
\begin{itemize}
\item[(i)] a particle at time $k$ and position $n\neq0$, just follows what the $S_{k+1}-S_k$ tells him (goes to $n+1$ if $S_{k+1}-S_k=1$ and to $n-1$ if $S_{k+1}-S_k=-1$),
\item[(ii)]  a particle at $0$ at time $k$ does not move if $S_{k+1}-S_k=-1$, and moves according to $\eta_k$ if $S_{k+1}-S_k=1$.
\end{itemize}
\noindent The situation on a finite half-lines is very close. Let $S=(S_n)_{n\in\Z}$ be a simple random walk on ${\Z}$, that is $(S_n)_{n\in\N}$ and $(S_{-n})_{n\in\N}$ are two independent simple random walks on $\Z$ and $(\vec{\eta}_i)_{i\in{\Z}}$ be a sequence of i.i.d random variables with law $\displaystyle{\sum_ {i=1}^{N}\alpha_i}\delta_{\vec{e}_i}$ which is independent of $S$. For $p\leq n$, set
$$S_{p,n}=S_n-S_p,\ S^{+}_{p,n}=S_n-\min_{h\in[p,n]} S_h=S_{p,n}-\min_{h\in[p,n]} S_{p,h}.$$
and for $p\in\Z, x\in G_{\N}$, define 
$$\varPsi_{p,p+1}(x)=x+\vec{e}(x)S_{p,p+1}\ \textrm{if}\ \ x\neq 0, \varPsi_{p,p+1}(0)=\vec{\eta}_p S^{+}_{p,p+1}.$$
$$K_{p,p+1}(x)=\delta_{x+\vec{e}(x)S_{p,p+1}}\ \textrm{if}\ \ x\neq 0, K_{p,p+1}(0)=\sum_{i=1}^{N}\alpha_i\delta_{S^{+}_{p,p+1}{\vec{e}_i}}.$$
In particular, we have $K_{p,p+1}(x)=E[\delta_{\varPsi_{p,p+1}(x)}|\sigma(S)]$. Now we extend this definition for all $p\leq n\in\Z, x\in G_{\N}$ by setting
$$\varPsi_{p,n}(x)=x1_{\{p=n\}}+\varPsi_{n-1,n}\circ \varPsi_{n-2,n-1}\circ\cdots\circ \varPsi_{p,p+1}(x)1_{\{p>n\}},$$
$$K_{p,n}(x)=\delta_x 1_{\{p=n\}}+K_{p,p+1}\cdots K_{n-2,n-1}K_{n-1,n}(x)1_{\{p>n\}}.$$
We equip $\mathcal P(G)$ with the following topology of weak convergence:
$$\beta (P, Q)=\sup\left\{|\int gdP- \int gdQ|, \| g\|_\infty+\displaystyle{\sup_{x \not =y}}\frac {|g(x)-g(y)|}{|x-y|}\leq 1, g(0) = 0\right\}.$$
In this paper, starting from $(\varPsi,K)$, we construct $(\varphi,K^W)$ and in particular show the following 
\begin{theorem}\label{lamjad}
(1) $\varPsi$ (resp. $K$) is a discrete flow of mappings (resp. kernels) on $G_{\mathbb N}$.\\
(2) There exists a joint realization $(\psi,N,\varphi,K^W)$ on a common probability space $(\Omega,\mathcal A,\mathbb P)$ such that\\
(i) $(\psi,N)\overset{law}{=}(\varPsi,K)$.\\
(ii) $(\varphi,W)$ (resp. $(K^W,W)$) is the unique flow of mappings (resp. Wiener flow) which solves $(T)$.\\
(iii) For all $ s\in\R, T>0, x\in G, x_n\in \frac{1}{\sqrt{n}}G_{\mathbb N}$ such that $\lim_{n\rightarrow\infty}x_n=x$, we have $$\lim_{n\rightarrow\infty}\sup_{s\leq t\leq s+T}|\frac{1}{\sqrt{n}}\psi_{\lfloor ns \rfloor,\lfloor nt\rfloor}(\sqrt{n}x_n)-\varphi_{s,t}(x)|=0\ \ a.s.$$
and 
\begin{equation}\label{thh}
\lim_{n\rightarrow\infty}\sup_{s\leq t\leq s+T}\beta(K_{\lfloor ns \rfloor,\lfloor nt \rfloor}(\sqrt{n}x_{n})(\sqrt{n}.),K^W_{s,t}(x))=0 \ \ a.s.
\end{equation}
\end{theorem}
\noindent This theorem implies also the following
\begin{corl}\label{bess}
For all $s\in\R, x\in G_{\N}$, let $t\longmapsto\varPsi(t)$ be the linear interpolation of $\left(\varPsi_{\lfloor ns \rfloor,k}(x), k\geq \lfloor ns \rfloor\right)$ and 
$ \varPsi_{s,t}^{n}(x):=\frac{1}{\sqrt n}\varPsi(nt),\ K^n_{s,t}(x)=E[\delta_{\varPsi_{s,t}^{n}(x)}|\sigma(S)], t\geq s, n\geq 1.$
For all $1\leq p\leq q$, $(x_i)_{1\leq i\leq q}\subset G$, let $x^n_i\in \frac{1}{\sqrt{n}}G_{\mathbb N}$ such that $\lim_{n\rightarrow\infty}x^n_i=x_i$. Define
$$Y^n=\left(\varPsi^n_{s_1,\cdot}(\sqrt{n}x^n_1),\cdots,\varPsi^n_{s_p,\cdot}(\sqrt{n}x_p^n),K_{s_{p+1},\cdot}^n(\sqrt{n}x^{n}_{p+1}),\cdots,K_{s_{q},\cdot}^n(\sqrt{n}x^n_q)\right).$$
Then $$Y^n\xrightarrow[\text{$n\rightarrow +\infty $}]{\text{law}}Y\ \ \textrm{in}\ \ \displaystyle\prod_{i=1}^{p}C([s_i,+\infty[,G)\times \displaystyle\prod_{j=p+1}^{q}C([s_j,+\infty[,\mathcal P(G))$$
where $$Y=\left(\varphi_{s_1,\cdot}(x_1),\cdots,\varphi_{s_p,\cdot}(x_p),K^W_{{s_{p+1}},\cdot}(x_{p+1}),\cdots,K^W_{s_q,\cdot}(x_q)\right).$$
\end{corl}
Our proof of Theorem \ref{lamjad} is based on a remarkable transformation introduced by Csaki and Vincze \cite{MR2168855} which is strongly linked with Tanaka's SDE. Let $S$ be a simple random walk on $\Z$ (SRW) and $\varepsilon$ be a Bernoulli random variable independent of $S$ (just one!). Then there exists a SRW $M$ such that 
$$\sigma(M)=\sigma(\varepsilon,S)$$
and moreover $$(\frac{1}{\sqrt{n}}{S}(nt),\frac{1}{\sqrt{n}}{M}(nt))_{t\geq 0}\xrightarrow[\text{$n\rightarrow +\infty $}]{\text{law}}(B_{t},W_{t})_{t\geq 0}\ \ \textrm{in}\ C([0,\infty[,\R^2).$$
where  $t\longmapsto S(t)$ (resp. $M(t)$) is the linear interpolation of $S$ (resp. $M$) and $B,W$ are two Brownian motions satisying Tanaka's equation $$dW_t=\textrm{sgn}(W_t)dB_t.$$
We will study this transformation with more details in Section 2 and then extend the result of Csaki and Vincze to Walsh Brownian motion (Proposition \ref{hajri11}); Let $S=(S_n)_{n\in\N}$ be a SRW and associate to $S$ the process $Y_n:=S_n-\displaystyle{\min_{k\leq n}} S_k$, flip independently every \textquotedblleft excursion \textquotedblright of $Y$ to each ray $D_i$ with probability $\alpha_i$, then the resulting process is not far from a random walk on $G$ whose law is given by (\ref{sirine}). 
In Section 3, we prove Proposition \ref{enri} and study the scaling limits of $\varPsi, K$.
\section{Csaki-Vincze transformation and consequences.} 
In this section, we review a relevant result of Csaki and Vincze and then derive some useful consequences offering a better understanding of Tanaka's equation.
\subsection{Csaki-Vincze transformation.}
\begin{theorem}\label{khali}(\cite{MR2168855} page 109)
 Let $S = (S_n)_{n\geq 0}$ be a SRW. Then, there exists a SRW $\overline {S} =(\overline {S}_n)_{n\geq 0}$ such that:$$\overline Y_n:=\displaystyle{\max_{k\leq n}} \overline S_k - \overline S_n\Rightarrow
|\overline Y_n -|S_n|| \leq 2 \ \ \forall n\in {\N}.$$
\end{theorem}
\noindent\textbf {Sketch of the proof.} Here, we just give the expression of $\overline {S}$ with some useful comments (see also the figures below). We \textbf{insist} that a careful reading of the pages 109 and 110 \cite{MR2168855} is recommended for the sequel. Let $X_i=S_i-S_{i-1}, i\geq1$ and define
$$\tau_{1}=\min{\{i> 0 : S_{i-1}S_{i+1}< 0\}}
,\ \tau_{l+1}=\min{\{i>\tau_{l} : S_{i-1}S_{i+1}< 0\}}\ \forall l\geq 1.$$
For $j\geq 1$, set
$$\overline X_j=\sum_{l\geq0}(-1)^{l+1}X_1X_{j+1}1_{\{\tau_{l}+1\leq j\leq\tau_{l+1}\}}.$$
Let $\overline S_0=0,\ \overline S_j=\overline X_1+\cdots+\overline X_j,\ j\geq 1$. Then, the theorem holds for $\overline S$. We call $T(S)=\overline S$ the Csaki-Vincze transformation of $S$.\\

\begin{figure}[h]
\begin{center}
\resizebox{10cm}{7cm}{\input{fig_tanaka2.pstex_t}}
\caption{$S$ and $\overline{S}$.}
\end{center}
\end{figure}

\begin{figure}[h]
\begin{center}
\resizebox{11cm}{7cm}{\input{fig_tanaka3.pstex_t}}
\caption{$|S|$ and $\overline{Y}$.}
\end{center}
\end{figure}
\noindent Note that $T$ is an even function, that is $T(S)=T(-S)$. As a consequence of $(iii)$ and $(iv)$ \cite{MR2168855} (page 110), we have
\begin{equation}\label{csak}
\tau_{l}=\min{\{n\geq 0 , \overline S_n=2l\}}\ \forall l\geq 1. 
\end{equation}
This entails the following
\begin{corl}\label{M}
(1) Let $S$ be a SRW and define $\overline S=T(S)$. Then 
\begin{enumerate}
 \item [(i)] For all $n\geq0$, we have $\sigma(\overline{S}_j, j\leq n)\vee \sigma(S_1)=\sigma(S_j, j\leq n+1)$.
 \item [(ii)] $S_1$ is independent of $\sigma(\overline{S})$.
\end{enumerate}
(2) Let $\overline S =(\overline S_k)_{k\geq 0}$ be a SRW. Then
\begin{enumerate}
 \item [(i)] There exists a SRW $S$ such that: $$\overline Y_n:=\displaystyle{\max_{k\leq n}} \overline S_k - \overline S_n\Rightarrow
|\overline Y_n -|S_n|| \leq 2 \ \ \forall n\in {\N}.$$
\item [(ii)] $T^{-1}\{\overline S\}$ is reduced to exactly two elements $S$ and $-S$ where $S$ is obtained by adding information to $\overline S$.
\end{enumerate}
\end{corl}
\begin{proof} (1) We retain the notations just before the corollary. (i) To prove the inclusion $\subset$, we only need to check that $\{\tau_{l}+1\leq j\leq\tau_{l+1}\}\in\sigma(S_h, h\leq n+1)$ for a fixed $j\leq n$. This is clear since $\{\tau_{l}=m\}\in\sigma(S_h, h\leq m+1)$ for all $l, m\in\N$. For all $1\leq j\leq n$, we have $X_{j+1}=\sum_{l\geq0}(-1)^{l+1}X_1\overline X_j1_{\{\tau_{l}+1\leq j\leq\tau_{l+1}\}}$. By (\ref{csak}), $\{\tau_{l}+1\leq j\leq\tau_{l+1}\}\in\sigma(\overline S_h, h\leq j-1)$ and so the inclusion $\supset$ holds. (ii) We may write
$$\tau_{1}=\min{\{i> 1 : X_1S_{i-1}X_1S_{i+1}< 0\}}
,\ \tau_{l+1}=\min{\{i>\tau_{l} : X_1S_{i-1}X_1S_{i+1}< 0\}}\ \forall l\geq 1.$$
This shows that $\overline S$ is $\sigma(X_1X_{j+1}, j\geq0)$-measurable and (ii) is proved.\\
(2) (i) Set $\overline{X}_j=\overline S_j-\overline S_{j-1}, j\geq1$ and $\tau_{l}=\min{\{n\geq 0 , \overline S_n=2l\}}$ for all $l\geq 1$. Let $\varepsilon$ be a random variable independent of $\overline S$ such that: $$\mathbb P(\varepsilon=1)=\mathbb P(\varepsilon=-1)=\frac{1}{2}.$$
Define
$$X_{j+1}=\varepsilon1_{\{j=0\}}+\left(\sum_{l\geq0}(-1)^{l+1}\varepsilon \overline X_{j}1_{\{\tau_{l}+1\leq j\leq\tau_{l+1}\}}\right)1_{\{j\geq 1\}}.$$
Then set $S_0=0, S_{j}=X_{1}+\cdots X_{j}, j\geq 1$. It is not hard to see that the sequence of the random times $\tau_i(S), i\geq1$ defined from $S$ as in Theorem \ref{khali} is exactly $\tau_i, i\geq1$ and therefore $T(S)=\overline S$. (ii) Let $S$ such that $T(S)=\overline S$. By (1), $\sigma(\overline{S})\vee \sigma(S_1)=\sigma(S)$ and $S_1$ is independent of $\overline S$ which proves (ii).
\end{proof}
\subsection{The link with Tanaka's equation.}
Let $S$ be a SRW, $\overline S=-T(S)$ and $t\longmapsto S(t)$ (resp. $\overline S(t)$) be the linear interpolation of $S$ (resp. $\overline S$) on $\R$. Define for all $n\geq 1$, $S_{t}^{(n)}=\frac{1}{\sqrt n}S(nt), \overline{S}_{t}^{(n)}=\frac{1}{\sqrt n}\overline{S}(nt)$. Then, it can be easily checked (see Proposition 2.4 in \cite{MR838085} page 107) that
$$(\overline{S}_{t}^{(n)},S_{t}^{(n)})_{t\geq 0}\xrightarrow[\text{$n\rightarrow +\infty $}]{\text{law}}(B_{t},W_{t})_{t\geq 0}\ \ \textrm{in}\ C([0,\infty[,\R^2).$$
In particular $B$ and $W$ are two standard Brownian motions. On the other hand, $|Y_n^{+} -|S_n|| \leq 2\  \forall n\in {\N}\ \textrm{with} \ \ Y_n^{+}:=\overline S_n-\displaystyle{\min_{k\leq n}} \overline S_k $ by Theorem \ref{khali} which implies $|W_t|=B_t-\displaystyle{\min_{0\leq u \leq t}} B_u$. Tanaka's formula for local time gives
$$|W_t|=\int_{0}^{t}sgn(W_u)dW_u+L_t(W)=B_t-\displaystyle{\min_{0\leq u \leq t}} B_u,$$
where $L_t(W)$ is the local time at $0$ of $W$ and so 
\begin{equation}\label{houssem}
 dW_u=sgn(W_u)dB_u.
\end{equation}

We deduce that for each SRW $S$ the couple $(-T(S),S)$, suitably normalized and time scaled converges in law towards $(B,W)$ satisfying (\ref{houssem}). Finally, remark that $-T(S)=\overline S \Rightarrow -T(-S)=\overline S $ is the analogue of $W \textrm{solves}\  (\ref{houssem}) \Rightarrow -W \textrm{solves}\  (\ref{houssem})$.
We have seen how to construct solutions to (\ref{houssem}) by means of $T$. In the sequel, we will use this approach to construct a stochastic flow of mappings which solves equation $(T)$ in general.
\subsection{Extensions.}
Let $S = (S_n)_{n\geq 0}$ be a SRW and set $Y_n:=\displaystyle{\max_{k\leq n}} S_k-S_n$. 
For $0\leq p<q$, we say that $E=[p,q]$ is an excursion for $Y$ if the following conditions are satisfied (with the convention $Y_{-1}=0$):\\
$\bullet$\ $Y_p=Y_{p-1}=Y_q=Y_{q+1}=0$. \ \ \\
$\bullet$\ $\forall\ p\leq j<q, Y_j=0\Rightarrow Y_{j+1}=1$.\\
For example in Figure 3, $[2,14], [16,18]$ are excursions for $\overline{Y}$. If $E=[p,q]$ is an excursion for $Y$, define $e(E):=p,\ f(E):=q$. \\
Let $(E_{i})_{i\geq 1}$ be the random set of all excursions of $Y$ ordered such that: $e(E_{i})<e(E_{j}) \ \forall i< j$. From now on, we call $E_{i}$ the $i$th excursion of $Y$. Then, we have
\begin{prop}\label{hajri11}On a probability space $(\Omega ,\mathcal{A},P )$, consider the following jointly independent processes:\\
$\bullet$\ $\vec{\eta}=(\vec{\eta}_i)_{i\geq 1}$, a sequence of i.i.d random variables distributed according to $\displaystyle{\sum_ {i=1}^{N}\alpha_i}\delta_{\vec{e}_i}$.\\ 
$\bullet$ \ $(\overline S_n)_{n\in \N} $ a SRW.\\
Then, on an extension of $(\Omega ,\mathcal{A},P )$, there exists a Markov chain $(M_n)_{n\in \N}$ started at $0$ with stochastic matrix given by (\ref{sirine}) such that: $$\overline Y_n:=\displaystyle{\max_{k\leq n}} \overline S_k-\overline S_n\Rightarrow|M_n-\vec{\eta}_i \overline Y_n|\leq 2$$ on the $i$th excursion of $\overline Y$.
\end {prop}
\noindent\begin{proof}
Fix $S\in T^{-1}\{\overline S\}$. Then, by Corollary \ref{M}, we have $|\overline Y_n -|S_n|| \leq 2 \ \ \forall n\in {\N}$. Consider a sequence $(\vec{\beta}_i)_{i\geq 1}$  of i.i.d random variables distributed according to $\displaystyle{\sum_ {i=1}^{N}\alpha_i}\delta_{\vec{e}_i}$ which is independent of $(\overline S,\vec\eta)$. Denote by $(\tau_{l})_{l\geq 1}$ the sequence of random times constructed in the proof of Theorem \ref{khali} from $S$. It is sufficient to look to what happens at each interval $[\tau_{l},\tau_{l+1}]$ (with the convention $\tau_{0}=0$). \\
Using (\ref{csak}), we see that in $[\tau_{l},\tau_{l+1}]$ there are two jumps of $\displaystyle{\max_{k\leq n}} \overline S_k$; from $2l$ to $2l+1$ ($J_1$) and from $2l+1$ to $2l+2$ ($J_2$). The last jump ($J_2$) occurs always at $\tau_{l+1}$ by (\ref{csak}). Consequently there are only 3 possible cases:\\
 (i) There is no excursion of $\overline Y$ ($J_1$ and $J_2$ occur respectively at $\tau_l+1$ and $\tau_l+2$, see $[0,\tau_1]$ in Figure 3).\\
 (ii) There is just one excursion of $\overline Y$ (see $[\tau_1,\tau_2]$ in Figure 3).\\
 (iii) There are 2 excursions of $\overline Y$ (see $[\tau_2,\tau_3]$ in Figure 3).\\
 Note that: $\overline Y_{\tau_{l}}=\overline Y_{\tau_{l+1}}=S_{\tau_{l}}=S_{\tau_{l+1}}=0$.
In the case (i), we have necessarily $\tau_{l+1}=\tau_{l}+2$. Set $M_n=\vec{\beta}_{l}.|S_n|$ \ \ $\forall n\in[\tau_{l},\tau_{l+1}]$.\\
To treat other cases, the following remarks may be useful: from the expression of $\overline S$, we have $\forall l\geq 0$
\begin{itemize}
\item[(a)]If $k\in [\tau_{l}+2,\tau_{l+1}]$, $\overline S_{k-1}=2l+1\Longleftrightarrow S_{k}=0$.
\item[(b)]If $k\in[\tau_{l},\tau_{l+1}]$, $\overline Y_{k}=0\Rightarrow |S_{k+1}|\in \{0,1\}\ \textrm{and}\ S_{k+1}=0\Rightarrow \overline Y_{k}=0$.
\end{itemize}
In the case (ii), let $E_l^{1}$ be the unique excursion of $\overline Y$ in the interval $[\tau_{l},\tau_{l+1}]$. Then, we have two subcases:\\
(ii1) $f(E_l^{1})=\tau_{l+1}-2$ ($J_1$ occurs at $\tau_{l+1}-1$).\\
If $\tau_l +2 \leq k\leqslant f(E_l^{1})+1$, then $k-1\leqslant f(E_l^{1})$, and so $\overline S_{k-1}\neq 2l+1$. Using (a), we get: $S_k\neq 0$. Thus, in this case the first zero of $S$ after $\tau_{l}$ is $\tau_{l+1}$. Set: $M_n=\vec{\eta}_{N(E_l^{1})}.|S_n|$, where $N(E)$ is the number of the excursion $E$.\\
(ii2)$f(E_l^{1})=\tau_{l+1}-1$ ($J_1$ occurs at $\tau_l+1$ and so $\overline Y_{\tau_l+1}=0)$). In this case, using (b) and the figure below we see that the first zero $\tau_{l}^{*}$ of $S$ after $\tau_{l}$ is $e(E_l^{1})+1=\tau_l+2$.

\begin{figure}[h]
\begin{center}
\resizebox{6cm}{2cm}{\input{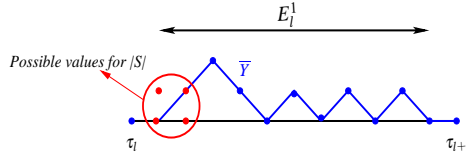}}
\caption{The case (ii2).}
\end{center}
\end{figure}

Set $$\begin{array}{ll}
M_n=\begin{cases}
\ \vec{\beta}_l.|S_n|&\text{if }\ n\in [\tau_{l},\tau_{l}^{*}-1]\\
\ \vec{\eta}_{N(E_l^1)}.|S_n|&\text {if }\ n\in[\tau_{l}^{*},\tau_{l+1}]\\
\end{cases}
\end{array}$$ \\
In the case (iii), let $E_l^{1}$ and $E_l^{2}$ denote respectively the first and $2$nd excursion of $\overline Y$ in $[\tau_{l},\tau_{l+1}]$. We have,  $\tau_{l}+2 \leq k\leq e(E_l^{2})\Rightarrow k-1\leq e(E_l^{2})-1=f(E_l^{1})\Rightarrow \overline S_{k-1}\neq 2l+1\Rightarrow S_k\neq 0 $ by (a). Hence, the first zero of $S$ after  $\tau_{l}$ is $\tau_{l}^{*}:=e(E_l^{2})+1$ using $\overline Y_{k}=0\Rightarrow |S_{k+1}|\in \{0,1\}$ in (b). Set:
$$\begin{array}{ll}
M_n=\begin{cases}
\ \vec{\eta}_{N(E_l^{1})}.|S_n|&\text{if}\ n\in[\tau_{l},\tau_{l}^{*}-1]\\
\ \vec{\eta}_{N(E_l^{2}}).|S_n|&\text {if}\ n\in[\tau_{l}^{*},\tau_{l+1}]\\
\end{cases}
\end{array}$$
Let ($M_n)_{n\in \N}$ be the process constructed above. Then clearly $|M_n-\vec{\eta}_i \overline Y_n|\leq 2$ on the $i$th excursion of $\overline Y$.\\
To complete the proof, it suffices to show that the law of $(M_n)_{n\in {\N}}$ is given by (\ref{sirine}). The only point to verify is $\mathbb P(M_{n+1}=\vec{e}_i|M_n=0)=\alpha_i$. For this, consider on another probability space the jointly independent processes $(S,\vec{\gamma},\vec{\lambda})$ such that $S$ is a SRW and $\vec{\gamma},\vec{\lambda}$ have the same law as $\vec\eta$. Let $(\tau_{l})_{l\geq 1}$ be the sequence of random times defined from $S$ as in Theorem \ref{khali}. For all $l\in {\N}$, denote by $\tau_{l}^{*}$ the first zero of $S$ after $\tau_{l}$ and set
\\$$\begin{array}{ll}
V_n=\begin{cases}
\ \vec{\gamma}_l.|S_n|&\text{if }\ n\in [\tau_{l},\tau_{l}^{*}-1]\\
\ \vec{\lambda}_{l}.|S_n|&\text {if }\ n\in[\tau_{l}^{*},\tau_{l+1}]\\
\end{cases}
\end{array}$$
It is clear, by construction, that $M\overset{law}{=}V$. We can write:
$$\{\tau_{0},\tau_{0}^{*},\tau_{1},\tau_{1}^{*},\tau_{2},\cdots\}=\{T_0,T_1,T_2,\cdots\} \ \ \textrm{with}\  T_0=0<T_1<T_2<\cdots.$$
For all $k\geq 0$, let $\vec\zeta_k:=\displaystyle{\sum_{j=0}^{N}\vec{e}_j {1}_{\{V|_{[T_k,T_{k+1}]}\in D_j\}}}$. Obviously, $S$ and $\vec\zeta_k$ are independent and $\vec\zeta_k\overset{law}{=}\displaystyle{\sum_ {i=1}^{N}\alpha_i}\delta_{\vec{e}_i}$. Furthermore
\begin{eqnarray}
\mathbb P(V_{n+1}=\vec{e}_i|V_n=0)&=&\frac{1}{\mathbb P(S_n=0)} \displaystyle{\sum_{k=0}^{+\infty}\mathbb P(V_{n+1}=\vec{e}_i,S_n=0,n\in [T_k,T_{k+1}[)}\nonumber\\
&=&\frac{1}{\mathbb P(S_n=0)} \displaystyle{\sum_{k=0}^{+\infty}\mathbb P(\vec\zeta_k=\vec{e}_i,S_n=0,n\in [T_k,T_{k+1}[)}\nonumber\\
&=&\alpha_i\nonumber\
\end{eqnarray}
This completes the proof of the proposition.\end{proof}
\begin{rem}
With the notations of Proposition \ref{hajri11}, let $(\vec\eta.\overline Y)$ be the Markov chain defined by $(\vec\eta.\overline Y)_n=\vec\eta_i \overline Y_n$ on the ith excursion of $\overline Y$ and
$(\vec\eta.\overline Y)_n=0$ if $\overline Y_n=0$. Then the stochastic Matrix of $(\vec\eta.\overline Y)$  is given by 
\begin{equation}\label{srrr}
M(0,0)=\frac{1}{2}, M(0,\vec{e}_{i})=\frac{\alpha_i}{2}, M(n\vec{e}_{i},(n+1)\vec{e}_{i})=M(n\vec{e}_{i},(n-1)\vec{e}_{i})=\frac{1}{2}, i\in [1,N],\  n\in {\N^{*}}.
\end{equation}
\end{rem}
\section{Proof of main results.}
\subsection{Proof of Proposition \ref{enri}.}
Let $(Z_t)_{t\geq 0}$ be a $W(\alpha_1,\cdots,\alpha_N)$ on $G$ started at $0$. For all $i\in [1,N]$, define $Z_{t}^i = |Z_t| 1_{\{Z_t\in D_i\}}-|Z_{t}| 1_{\{Z_t\notin D_i\}}$. Then $Z^i_t=\varPhi^i(Z_t)$ where $\varPhi^i(x)=|x|1_{\{x\in D_i\}}-|x|1_{\{x\notin D_i\}}$. Let $Q^i$ be the semigroup of the skew Brownian motion of parameter $\alpha_i$ $(SBM(\alpha_i))$ (see \cite{MR1725357} page 87). Then the following relation is easy to check: $P_{t}(f\circ\varPhi^i)=Q_t^if\circ\varPhi^i$ for all bounded measurable function $f$ defined on $\R$. This shows that $Z^i$ is a $SBM(\alpha_i)$ started at $0$. For $n\geq1$, $i\in[1,N]$, define  $$T^n_0=0,\ \ T^n_{k+1}=\inf\{r\geq0: d(Z_r,Z_{T^n_k})=\frac{1}{\sqrt{n}}\}, k\geq0.$$
$$T^{n,i}_0=0,\ \ T^{n,i}_{k+1}=\inf\{r\geq0: |Z^i_r-Z^i_{T^{n,i}_k}|=\frac{1}{\sqrt{n}}\}, k\geq0.$$
Remark that $T^n_{k+1}=T^{n,i}_{k+1}=\inf\{r\geq0: ||Z_r|-|Z_{T^n_k}||=\frac{1}{\sqrt{n}}\}$. Furthermore if $Z_t\in D_i$, then obviously $d(Z_t,Z_s)=|Z^i_t-Z^i_s|$ for all $s\geq0$ and consequently 
\begin{equation}\label{machin}
 d(Z_t,Z_s)\leq\sum_{i=1}^{N}|Z^i_t-Z^i_s|.
\end{equation}
Now define $Z^n_k=\sqrt{n}Z_{T^n_k}, Z^{n,i}_k=\sqrt{n}Z^i_{T^{n,i}_k}$. Then $(Z^n_k,k\geq0)\overset{law}{=}M$ (see the proof of Proposition 2 in \cite{MR5010101}). For all $T>0$, we have
$$\sup_{t\in[0,T]}d(Z_t,\frac{1}{\sqrt{n}}Z^n_{\lfloor nt\rfloor})\leq \sum_{i=1}^{N}\sup_{t\in[0,T]}|Z^i_t-\frac{1}{\sqrt{n}}Z^{n,i}_{\lfloor nt\rfloor}|\xrightarrow[\text{$n\rightarrow +\infty $}]{}0\ \textrm{in probability}$$
by Lemma 4.4 \cite{MR1975425} which proves our result.
\begin{rems}
(1) By (\ref{machin}), a.s. $t\mapsto Z_t$ is continuous. We will always suppose that Walsh Brownian motion is continuous.\\
(2) By combining the two propositions \ref{enri} and \ref{hajri11}, we deduce that $(\vec\eta.\overline Y)$ rescales as Walsh Brownian motion in the space of continuous functions. It is also possible to prove this result by showing that the family of laws is tight and that any limit process along a subsequence is the Walsh Brownian motion.
\end{rems}

\subsection{Scaling limits of $(\varPsi,K)$.}
Set $\vec{\eta}_{p,n}=\vec{e}(\varPsi_{p,n})$ for all $p\leq n$ where $\varPsi_{p,n}=\varPsi_{p,n}(0)$. 
\begin{prop}\label{thar}
\begin{enumerate}
\item [(i)] For all $p\leq n,$ $|\varPsi_{p,n}|=S^{+}_{p,n}$.
\item[(ii)] For all $p<n<q$, 
$$\mathbb P(\vec{\eta}_{p,q}=\vec{\eta}_{n,q}|\textrm{min}_{h\in[p,q]}S_h=\textrm{min}_{h\in[n,q]}S_h)=1$$
and 
$$\mathbb P(\vec{\eta}_{p,n}=\vec{\eta}_{p,q}|\textrm{min}_{h\in[p,n]}S_h=\textrm{min}_{h\in[p,q]}S_h, S^+_{p,j}>0\ \forall j\in[n,q])=1.$$
\item[(iii)] Set $T_{p,x}=\inf\{q\geq p: S_q-S_p=-|x|\}$. Then for all $p\leq n$, $x\in G_{\N}$,
$$\varPsi_{p,n}(x)=(x+\vec{e}(x)S_{p,n}){1}_{\{n\leq\ T_{p,x}\}}+\varPsi_{p,n}{1}_{\{n>T_{p,x}\}};$$
$$K_{p,n}(x)=E[\delta_{\varPsi_{p,n}(x)}|\sigma(S)]=\delta_{x+\vec{e}(x)S_{p,n}} 1_{\{n\leq T_{p,x}\}}+\sum_ {i=1}^{N}\alpha_i\delta_{S^+_{p,n}\vec{e}_i}1_{\{n>T_{p,x}\}}.$$
\end{enumerate}
\end{prop}
\noindent\begin{proof} (i) We take $p=0$ and prove the result by induction on $n$. For $n=0$, this is clear. Suppose the result holds for $n$. If $\varPsi_{0,n}\in G^{*}$, then $S^+_{0,n}>0$ and so $\textrm{min}_{h\in[0,n]}S_h=\textrm{min}_{h\in[0,n+1]}S_h$. Moreover 
$\varPsi_{0,n+1}=\varPsi_{0,n}+\vec{\eta}_{0,n}S_{n,n+1}=(S_{n+1}-\textrm{min}_{h\in[0,n]}S_h)\vec{\eta}_{0,n}=S^+_{0,n+1}\vec{\eta}_{0,n}$. If $\varPsi_{0,n}=0$, then $S^{+}_{0,n}=0$ and $|\varPsi_{0,n+1}|=S^{+}_{n,n+1}$. But $\textrm{min}_{h\in[0,n+1]}S_h=\textrm{min}(\textrm{min}_{h\in[0,n]}S_h,S_{n+1})=\textrm{min}(S_n,S_{n+1})$ since $S^{+}_{0,n}=0$ which proves (i).\\
(ii) Let $p<n<q$. If $\textrm{min}_{h\in[p,q]}S_h=\textrm{min}_{h\in[n,q]}S_h$, then $S^+_{p,q}=S^+_{n,q}$. When $S^+_{p,q}=0$, we have
$\vec{\eta}_{p,q}=\vec{\eta}_{n,q}=\vec{e}_N$ by convention. Suppose that $S^+_{p,q}>0$, then clearly 
$$J:=\sup\{j<q: S^+_{p,j}=0\}=\sup\{j<q: S^+_{n,j}=0\}.$$
By the flow property of $\varPsi$, we have $\varPsi_{p,q}=\varPsi_{n,q}=\varPsi_{J,q}$. The second assertion of (ii) is also clear.\\
(iii) By (i), we have $\varPsi_{p,n}=\varPsi_{p,n}(x)=0$ if $n=T_{p,x}$ and so $\varPsi_{p,.}(x)$ is given by $\varPsi_{p,.}$ after $T_{p,x}$ using the cocyle property. The last claim is easy to establish.
\end{proof}
For all $s\in\R$, let $d_s$ (resp.  $d_{\infty}$) be the distance of uniform convergence on every compact subset of $C([s, +\infty[,G)$ (resp. $C(\R,\R)$). Denote by $\mathbb D= \{s_n, n\in \N\}$ the set of all dyadic numbers of $\R$ and define $\widetilde{C} = C(\R,\R)\times \displaystyle\prod_{n=0}^{+\infty}C([s_n, +\infty[,G)$ equipped with the metric:
$$d(x,y) =d_{\infty}(x',y')+\sum_{n=0}^{+\infty }\frac{1}{2^n} \inf(1,d_{s_n}(x_n,y_n))\ \ \textrm{where}\ x=(x',x_{s_0},\cdots), y=(y',y_{s_0},\cdots).$$
Let $t\longmapsto S(t)$ be the linear interpolation of $S$ on $\R$ and define $S_{t}^{(n)}=\frac{1}{\sqrt n}S(nt), n\geq 1$.
If $u\leq0$, we define $\lfloor u \rfloor=-\lfloor -u \rfloor$. Then, we have
$$S_{t}^{(n)}= S_{t}^{n}+o(\frac{1}{\sqrt n}),\ \textrm{with}\ \ S_{t}^{n}:=\frac{1}{\sqrt n}S_{\lfloor nt \rfloor}.$$ 
Let $\varPsi_{s,t}^n=\varPsi_{s,t}^n(0)$ (defined in Corollary \ref{bess}). Then $\varPsi_{s,t}^{(n)}:=\frac{1}{\sqrt n}\varPsi_{\lfloor ns \rfloor,\lfloor nt \rfloor}+o(\frac{1}{\sqrt n})$ and we have the following
\begin{lemma}\label{tight}
Let $\mathbb P_n$ be the law of $Z^n =({S_.}^{(n)},(\varPsi_{s_i,.}^{(n)})_{i\in \N})$ in $\widetilde{C}$. Then $(\mathbb P_n, n\geq1)$ is tight.
\end{lemma}
\begin{proof} By Donsker theorem $\mathbb P_{S^{(n)}}\longrightarrow\mathbb P_{W}$ in $C(\R,\R)$ as $n\rightarrow\infty$ where $\mathbb P_{W}$ is the law of any Brownian motion on $\R$. Let $\mathbb P_{Z_{s_i}}$ be the law of any $W(\alpha_1,\cdots,\alpha_N)$ started at $0$ at time $s_i$. Plainly, the law of $\varPsi_{p,p+.}$ is given by (\ref{srrr}) and so by Propositions \ref{enri} and \ref{hajri11}, for all $i\in\N$, $\mathbb P_{\varPsi_{s_i,.}^{(n)}}\longrightarrow\mathbb P_{Z_{s_i}}$ in $C([s_i,+\infty[,G)$ as $n\rightarrow\infty$. Now the lemma holds using Proposition 2.4 \cite{MR838085} (page 107).  
 \end{proof}
Fix a sequence $(n_k,k\in\N)$ such that $Z^{n_k}\xrightarrow[\text{$k\rightarrow +\infty $}]{\text{law}}Z$ in $\widetilde{C}$. In the next paragraph, we will describe the law of $Z$. Notice that $(\varPsi_{p,n})_{p\leq n}$ and $S$ can be recovered from $(Z^{n_k})_{k\in\N}$. 
Using Skorokhod representation theorem, we may assume that $Z$ is defined on the original probability space and the preceding convergence holds almost surely.  Write $Z = (W,\psi_{s_1,.},\psi_{s_2,.},\cdots)$. Then, $(W_{t})_{t\in\R}$ is a Brownian motion on $\R$ and $(\psi_{s,t})_{t\geq s}$ is an $W(\alpha_1,\cdots,\alpha_N)$ started at $0$ for all $s\in \mathbb D$.
\subsubsection{Description of the limit process.} 
 Set $\vec{\gamma}_{s,t}=\vec{e}(\psi_{s,t}), s\in\mathbb D, s<t$ and define $\min_{u,v}=\min_{r\in[u,v]} W_r$, $u\leq v\in\R$. Then, we have
\begin{prop}\label{tharwet}
\begin{enumerate}
\item[(i)] For all $s\leq t, s\in  \mathbb D$, $|\psi_{s,t}|=W^{+}_{s,t}$.
\item[(ii)] For all $s<t, u<v,\ s, u \in  \mathbb D$, 
$$\mathbb P(\vec\gamma_{s,t} =\vec\gamma _{u,v}|\textrm{min}_{s,t}=\textrm{min}_{u,v})=1\ \textrm{if}\ \ \mathbb P({\textrm{min}_{s,t}}={\textrm{min}_{u,v}})>0.$$
\end{enumerate}
\end{prop} 
\begin{proof}
(i) is immediate from the convergence of $Z^{n_k}$ towards $Z$ and Proposition \ref{thar} (i). (ii) We first prove that for all $s<t<u$, 
\begin{equation}\label{zez}
\mathbb P(\vec{\gamma}_{s,u}=\vec{\gamma}_{t,u}|\textrm{min}_{s,u}=\textrm{min}_{t,u})=1\ \textrm{if}\ s, t\in\mathbb D
\end{equation}
and 
\begin{equation}\label{zeze}
\mathbb P(\vec{\gamma}_{s,t}=\vec{\gamma}_{s,u}|\textrm{min}_{s,t}=\textrm{min}_{s,u})=1 \ \textrm{if}\ s\in\mathbb D.
\end{equation}
Fix  $s<t<u$ with $s, t\in\mathbb D$ and let show that a.s. 
\begin{equation}\label{jdid}\{\textrm{min}_{s,u}=\textrm{min} _{t,u}\}\subset\{\exists k_0,\ \ \vec\eta_{\lfloor n_ks \rfloor,\lfloor n_ku \rfloor}=\vec\eta_{\lfloor n_kt \rfloor,\lfloor n_ku \rfloor}\ \textrm{for all}\  k\geq k_0\}.
\end{equation}
We have $\{\textrm{min}_{s,u}=\textrm{min} _{t,u}\}=\{\textrm{min}_{s,t}<\textrm{min} _{t,u}\}\ \ a.s.$ By uniform convergence the last set is contained in $$\{\exists k_0, \min_{\lfloor n_ks \rfloor\leq j\leq\lfloor n_kt \rfloor} S_j<\min_{ \lfloor n_kt \rfloor\leq j\leq\lfloor n_ku \rfloor} S_j\ \textrm{for all}\ \ k\geq k_0\}$$
which is a subset of $$\{\exists k_0, \min_{\lfloor n_ks \rfloor\leq j\leq\lfloor n_ku \rfloor} S_j=\min_{\lfloor n_kt \rfloor\leq j\leq\lfloor n_ku \rfloor} S_j\ \textrm{for all}\ k\geq k_0\}.$$
This gives (\ref{jdid}) using Proposition \ref{thar} (ii). Since $x\longrightarrow\vec{e}(x)$ is continuous on $G^*$, on $\{\textrm{min}_{s,u}=\textrm{min}_{t,u}\}$, we have $$\vec{\gamma}_{s,u}=\lim_{k\rightarrow\infty}\vec{e}(\frac{1}{\sqrt n_k}\varPsi_{\lfloor n_ks \rfloor,\lfloor n_ku \rfloor})=\lim_{k\rightarrow\infty}\vec{e}(\frac{1}{\sqrt n_k}\varPsi_{\lfloor n_kt \rfloor,\lfloor n_ku \rfloor})=\vec{\gamma}_{t,u}\ \ a.s.$$
which proves (\ref{zez}). If $s\in\mathbb D, t>s$ and $\textrm{min} _{s,t}=\textrm{min} _{s,u}$, then $s$ and $t$ are in the same excursion interval of $W^+_{s,}$ and so $W^+_{s,r}>0$ for all $r\in[t,u]$. As preceded, $\{\textrm{min}_{s,t}=\textrm{min} _{s,u}\}$ is a.s. included in
$$\{\exists k_0, \min_{\lfloor n_ks \rfloor\leq j\leq\lfloor n_kt \rfloor} S_j=\min_{\lfloor n_ks \rfloor\leq j\leq\lfloor n_ku \rfloor} S_j\ , S^+_{\lfloor n_ks \rfloor,j}>0\ \forall j\in[\lfloor n_kt \rfloor,\lfloor n_ku \rfloor], k\geq k_0\}.$$
Now it is easy to deduce (\ref{zeze}) using Proposition \ref{thar} (ii). To prove (ii), suppose that $s\leq u, \textrm{min}_{s,t}=\textrm{min}_{u,v}$. There are two cases to discuss, (a) $s\leq u\leq v\leq t$, (b) $s\leq u\leq t\leq v$ (in any other case $\mathbb P({\textrm{min}_{s,t}}={\textrm{min}_{u,v}})=0$). In case (a), we have $\textrm{min}_{s,t}=\textrm{min}_{u,v}=\textrm{min}_{u,t}$ and so $\vec{\gamma}_{s,t}=\vec{\gamma}_{u,t}=\vec{\gamma}_{u,v}$ by (\ref{zez}) and (\ref{zeze}). Similarly in case (b), we have $\vec{\gamma}_{s,t}=\vec{\gamma}_{u,t}=\vec{\gamma}_{u,v}$.
\end{proof}
 
\begin{prop}\label{karama}
Fix $s<t, s\in\mathbb D, n\geq1$ and $\{(s_i,t_i); 1\leq i\leq n\}$ with $s_i<t_i, s_i\in\mathbb D$. Then 
\begin{enumerate}
\item[(i)] $\vec\gamma_{s,t}$ is independent of $\sigma(W)$.
\item[(ii)] For all $i\in[1,N]$, $h\in [1,n]$, we have
$$E[1_{\{\vec\gamma_{s,t}=\vec{e}_i\}}|(\vec\gamma_{s_i,t_i})_{1\leq i\leq n}, W]=1_{\{\vec\gamma_{s_h,t_h}=\vec{e}_i\}} \ \textrm{on}\ \{\textrm{min}_{s,t}=\textrm{min}_{s_h,t_h}\}.$$
\item[(iii)] The law of $\vec\gamma_{s,t}$ knowing $(\vec\gamma_{s_i,t_i})_{1\leq i\leq n}$ and $W$ is given by
$\displaystyle\sum_{i=1}^{N}\alpha_i\delta_{\vec{e}_i}$ when $\textrm{min}_{s,t}\notin \{\textrm{min}_{s_i,t_i}; 1\leq i\leq n\}$.
\end{enumerate}
This entirely describes the law of $(W,\psi_{s,\cdot},s\in\mathbb D)$ in $\widetilde C$ independently of $(n_k,k\in\N)$ and consequently $Z^{n}\xrightarrow[\text{$n\rightarrow +\infty $}]{\text{law}}Z$ in $\widetilde{C}$.
\end{prop}
\begin{proof}
(i) is clear. (ii) is a consequence of Proposition \ref{tharwet} (ii). (iii) Write $\{s,t,s_i,t_i, 1\leq i\leq n\}=\{r_k, 1\leq k\leq m\}$ with $r_j<r_{j+1}$ for all $1\leq j\leq m-1$. Suppose that $s=r_i, t=r_h$ with $i<h$. Then a.s. $\{\textrm{min}_{r_j,r_{j+1}}, i\leq j\leq h-1\}$ are distinct and it will be sufficient to show that $\vec\gamma_{s,t}$ is independent of $\sigma((\vec\gamma_{s_i,t_i})_{1\leq i\leq n},W)$ conditionally to $A=\{\textrm{min}_{s,t}=\textrm{min}_{r_j,r_{j+1}},\ \textrm{min}_{s,t}\neq\textrm{min}_{s_i,t_i}\ \textrm{for all}\ 1\leq i\leq n\}$ for $j\in [i,h-1]$. On $A$, we have $\vec\gamma_{s,t}=\vec\gamma_{r_j,r_{j+1}}$, $\{\textrm{min}_{s_i,t_i}, 1\leq i\leq n\}\subset \{\textrm{min}_{r_k,r_{k+1}}, k\neq j\}$ and so $\{\vec\gamma_{s_i,t_i}, 1\leq i\leq n\}\subset \{\vec\gamma_{r_k,r_{k+1}}, k\neq j\}$. Since $\vec\gamma_{r_1,r_{2}}, \cdots,\vec\gamma_{r_{m-1},r_{m}}, W$ are independent, it is now easy to conclude.
\end{proof}
In the sequel, we still assume that all processes are defined on the same probability space and that $Z^{n}\xrightarrow[\text{$n\rightarrow +\infty $}]{\text{a.s.}}Z$ in $\widetilde{C}$. In particular $\forall s\in \mathbb D, T>0$,
\begin{equation}\label{hlimaa} 
\lim\limits_{\substack{k \to +\infty}}\displaystyle\sup_{s\leq t\leq s+T}|\frac{1}{\sqrt k}\varPsi_{\lfloor ks \rfloor,\lfloor kt \rfloor}-\psi_{s,t}|=0\ \ a.s.
\end{equation}
\subsubsection{Extension of the limit process.} For a fixed $s< t$, $\textrm{min}_{s,t}$ is attained in $]s,t[$ a.s. By Proposition \ref{tharwet} (ii), on a measurable set $\Omega_{s,t}$ with probability 1, $\lim\limits_{\substack{s'\to s+,s'\in{\mathbb D}}} \vec\gamma_{s',t}$ exists. Define $\vec\varepsilon_{s,t}=\lim\limits_{\substack{s'\to s+,s'\in{\mathbb D}}} \vec\gamma_{s',t}$ on $\Omega_{s,t}$ and give an arbitrary value to $\vec\varepsilon_{s,t}$ on $\Omega_{s,t}^{c}$. Now, let $\varphi_{s,t}=\vec\varepsilon_{s,t}W^+_{s,t}$. Then for all $s\in\mathbb D, t>s$, $(\vec\varepsilon_{s,t},\varphi_{s,t})$ is a modification of $(\vec\gamma_{s,t},\psi_{s,t})$. For all $s\in\R$, $t>s,\ \ \varphi_{s,t}=\lim\limits_{\substack{n\to \infty}}\varphi_{s_n,t}\ a.s.$, where $s_n=\frac{{\lfloor 2^{n}s\rfloor}+1}{2^n}$ and therefore $(\varphi_{s,t})_{t\geq s}$ is an $W(\alpha_1,\cdots,\alpha_N)$ started at $0$. Again, Proposition \ref{tharwet} (ii) yields 
\begin{equation}\label{rhayem}
\forall s<t, u<v,\ \mathbb P(\vec\varepsilon_{s,t} =\vec\varepsilon_{u,v}|\textrm{min}_{s,t}=\textrm{min} _{u,v})=1\ \textrm{if}\ \ \mathbb P({\textrm{min}_{s,t}}={\textrm{min}_{u,v}})>0.
\end{equation}
Define: $$\varphi_{s,t}(x)=(x+\vec{e}(x)W_{s,t}){1}_{\{t\leq\tau_{s,x}\}}+\varphi_{s,t}{1}_{\{t>\tau_{s,x}\}}, s\leq t, x\in G,$$
 where $W_{s,t}=W_t-W_s$ and $\tau_{s,x}$ is given by (\ref{didi}).

\begin{prop}\label{valerie}
 Let $x\in G, x_n\in \frac{1}{\sqrt{n}}G_{\mathbb N}, \lim_{n\rightarrow\infty}x_n=x$, $s\in \R, T>0$. Then, we have
$$\lim\limits_{\substack{n \to +\infty}}\displaystyle\sup_{s\leq t\leq s+T}|\frac{1}{\sqrt n}\varPsi_{\lfloor ns \rfloor,\lfloor nt \rfloor}(\sqrt n x_n)-\varphi_{s,t}(x)|=0\ \ a.s.$$
\end{prop}
\begin{proof} 
Let $s'$ be a dyadic number such that $s< s'< s+T$. By (\ref{rhayem}), for $t>s'$:
$$\{\textrm{min}_{s,t}=\textrm{min}_{s',t}\}\subset \{\varphi_{s,t}=\varphi_{s',t}\} \ \ a.s.$$
and so, a.s.
$$\forall t>s', t\in\mathbb D;\ \{\textrm{min}_{s,t}=\textrm{min}_{s',t}\}\subset\{\varphi_{s,t}=\varphi_{s',t}\}.$$
If $t>s', \textrm{min}_{s,t}=\textrm{min}_{s',t}$ and $t_n\in\mathbb D, t_n\downarrow t$ as $n\rightarrow\infty$, then $\textrm{min}_{s,t_n}=\textrm{min}_{s',t_n}$ which entails that $\varphi_{s,t_n}=\varphi_{s',t_n}$ and therefore $\varphi_{s,t}=\varphi_{s',t}$ by letting $n\rightarrow\infty$. This shows that a.s.
$$\forall t>s';\ \{\textrm{min}_{s,t}=\textrm{min}_{s',t}\}\subset\{\varphi_{s,t}=\varphi_{s',t}\}\ .$$
As a result a.s.
\begin{equation}\label{buqet}
\forall s'\in \mathbb D\cap ]s,s+T[, \forall t>s';\ \{\textrm{min}_{s,t}=\textrm{min}_{s',t}\}\subset\{\varphi_{s,t}=\varphi_{s',t}\}.
\end{equation}
By standard properties of Brownian paths, a.s. $\displaystyle\textrm{min}_{s,s+T}\notin \{W_s,W_{s+T}\}$ and
$$\forall p\in\N^*; \displaystyle\textrm{min}_{s,s+\frac{1}{p}}<W_s,\ \displaystyle\textrm{min}_{s,s+\frac{1}{p}}\neq W_{s+\frac{1}{p}},\exists ! u_p\in]s,s+\frac{1}{p}[:\displaystyle\textrm{min}_{s,s+\frac{1}{p}}=W_{u_p}.$$ The reasoning below holds almost surely: Take $p\geq 1, \displaystyle\textrm{min}_{s,s+\frac{1}{p}}>\displaystyle\textrm{min}_{s,s+T}$. Let $\mathcal S_{p}\in]s,s+\frac{1}{p}[$: $\displaystyle\textrm{min}_{s,s+\frac{1}{p}}=W_{\mathcal S_{p}}$ and $s'$ be a (random) dyadic number in $]s,\mathcal S_{p}[$. Then $\textrm{min}_{s,s'}>\textrm{min}_{s',t}$ for all $t\in[\mathcal S_p,s+T]$. By uniform convergence: 
$$\exists n_0\in\N: \forall n\geq n_0,\ \forall \mathcal S_p\leq t\leq s+T,\ \min_{u\in [s,s']}S_{{\lfloor nu \rfloor}}>\min_{u\in [s',t]}S_{{\lfloor nu \rfloor}}\ \textrm{and so}\ \ \varPsi_{\lfloor ns' \rfloor,\lfloor nt \rfloor}=\varPsi_{\lfloor ns \rfloor,\lfloor nt \rfloor}.$$
Therefore for $n\geq n_0$, we have $$\sup_{\mathcal S_p\leq t\leq s+T}|\frac{1}{\sqrt n}\varPsi_{\lfloor ns \rfloor,\lfloor nt \rfloor}-\varphi_{s,t}|=\sup_{\mathcal S_p\leq t\leq s+T}|\frac{1}{\sqrt n}\varPsi_{\lfloor ns' \rfloor,\lfloor nt \rfloor}-\varphi_{s',t}|\ (\textrm{using (\ref{buqet}))}$$ and so
\begin{eqnarray}
\sup_{s\leq t\leq s+T}|\frac{1}{\sqrt n}\varPsi_{\lfloor ns \rfloor,\lfloor nt \rfloor}-\varphi_{s,t}|&\leq& \sup_{s\leq t\leq \mathcal S_p}|\frac{1}{\sqrt n}\varPsi_{\lfloor ns \rfloor,\lfloor nt \rfloor}-\varphi_{s,t}|+\sup_{\mathcal S_p\leq t\leq s+T}|\frac{1}{\sqrt n}\varPsi_{\lfloor ns \rfloor,\lfloor nt \rfloor}-\varphi_{s,t}|\nonumber\\
&\leq& \sup_{s\leq t\leq {s+\frac{1}{p}}}(\frac{1}{\sqrt n}S^+_{\lfloor ns \rfloor,\lfloor nt \rfloor}+W^+_{s,t})+\sup_{\mathcal S_p\leq t\leq s+T}|\frac{1}{\sqrt n}\varPsi_{\lfloor ns' \rfloor,\lfloor nt \rfloor}-\varphi_{s',t}|\nonumber\\
&\leq& \sup_{s\leq t\leq {s+\frac{1}{p}}}(\frac{1}{\sqrt n}S^+_{\lfloor ns \rfloor,\lfloor nt \rfloor}+W^+_{s,t})+\sup_{s'\leq t\leq s'+T}|\frac{1}{\sqrt n}\varPsi_{\lfloor ns' \rfloor,\lfloor nt \rfloor}-\varphi_{s',t}|.\nonumber\
\end{eqnarray}
From (\ref{hlimaa}), a.s.
$\forall u\in\mathbb D, \lim\limits_{\substack{n \to +\infty}}\displaystyle\sup_{u\leq t\leq u+T}|\frac{1}{\sqrt n}\varPsi_{\lfloor nu \rfloor,\lfloor nt \rfloor}-\varphi_{u,t}|=0$. By letting $n$ go to $+\infty$ and then $p$ go to $+\infty$, we obtain 
\begin{equation}\label{sout}
 \lim_{n\rightarrow\infty}\sup_{s\leq t\leq s+T}|\frac{1}{\sqrt n}\varPsi_{\lfloor ns \rfloor,\lfloor nt \rfloor}-\varphi_{s,t}|=0\ a.s.
\end{equation}
We now show that 
\begin{equation}\label{sse}
\lim\limits_{\substack{n \to +\infty}}\frac{1}{n} T_{\lfloor ns \rfloor,\sqrt n x_n}=\tau_{s,x}\ a.s.
\end{equation}
We have $$\frac{1}{n} T_{\lfloor ns \rfloor,\sqrt n x_n}=\inf\{r\geq \frac{\lfloor ns \rfloor}{n}: S^n_r-S^n_s=-|x_n|\}.$$
For $\epsilon>0$, from $$\lim_{n\rightarrow\infty}\sup_{u\in [\tau_{s,x},\tau_{s,x}+\epsilon]}|(S^{n}_u-S^{n}_s+|x_n|)-(W_{s,u}+|x|)|=0,$$
we get $$\lim_{n\rightarrow\infty}\inf_{u\in [\tau_{s,x},\tau_{s,x}+\epsilon]}(S^{n}_u-S^{n}_s+|x_n|)=\inf_{u\in [\tau_{s,x},\tau_{s,x}+\epsilon]}(W_{s,u}+|x|)<0$$
which implies $\frac{1}{n} T_{\lfloor ns \rfloor,\sqrt n x_n}<\tau_{s,x}+\epsilon$ for $n$ large. If $x=0$, $\frac{1}{n}T_{\lfloor ns \rfloor,\sqrt n x_n}\geq \frac{\lfloor ns \rfloor}{n}$ entails obviously (\ref{sse}). If $x\neq0$, then working in $[s,\tau_{s,x}-\epsilon]$ as before and using $\inf_{u\in [s,\tau_{s,x}-\epsilon]}(W_u-W_s+|x|)>0$, we prove that $\frac{1}{n} T_{\lfloor ns \rfloor,\sqrt n x_n}\leq\tau_{s,x}-\epsilon$ for $n$ large which establishes (\ref{sse}).\\
Now 
\begin{equation}\label{legall}
\displaystyle\sup_{s\leq t\leq s+T}|\frac{1}{\sqrt n}\varPsi_{\lfloor ns \rfloor,\lfloor nt \rfloor}(\sqrt n x_n)-\varphi_{s,t}(x)|\leq \displaystyle\sup_{s\leq t\leq s+T}Q_{s,t}^{1,n}+\displaystyle\sup_{s\leq t\leq s+T}Q_{s,t}^{2,n}
\end{equation}
where
$$Q_{s,t}^{1,n}=|(x_n+\vec{e}(x_n)(S^n_t-S^n_s)){1}_{\{\lfloor nt \rfloor\leq T_{\lfloor ns \rfloor,\sqrt n x_n}\}}-(x+\vec{e}(x)W_{s,t}){1}_{\{t\leq\tau_{s,x}\}}|,$$
$$Q_{s,t}^{2,n}=|\frac{1}{\sqrt n}\varPsi_{\lfloor ns \rfloor,\lfloor nt \rfloor}{1}_{\{\lfloor nt \rfloor>T_{\lfloor ns \rfloor,\sqrt n x_n}\}}-\varphi_{s,t}{1}_{\{t>\tau_{s,x}\}}|.$$
By (\ref{sout}), (\ref{sse}) and the convergence of $\frac{1}{\sqrt n}S_{\lfloor n. \rfloor}$ towards $W$ on compact sets, the right-hand side of (\ref{legall}) converges to $0$ when $n\rightarrow +\infty$.
\end{proof}
\begin{rem}
From the definition of $\vec\varepsilon_{s,t}$ (or Proposition \ref{valerie}), it is obvious that\\ $\vec\varepsilon_{r_1,r_2},\cdots,\vec\varepsilon_{r_{m-1},r_m},W$ are independent for all $r_1<\cdots<r_m$. Using (\ref{rhayem}), we easily check that (i), (ii) and (iii) of Proposition \ref{karama} are satisfied for all $s<t, n\geq1, \{(s_i,t_i); 1\leq i\leq n\}$ with $s_i<t_i$ (the proof remains the same as Proposition \ref{karama}).  
\end{rem}
\begin{prop}
 $\varphi$ is the unique stochastic flow of mappings solution of $(T)$.
\end{prop}
\begin{proof}
Fix $s<t<u, x\in G$ and let prove that $\varphi_{s,u}(x)=\varphi_{t,u}\circ\varphi_{s,t}(x)$ a.s.
We follow Lemma 4.3 \cite{MR2235172} and denote $\tau_{s,x}$ by $\tau_{s}(x)$. All the equalities below hold a.s.\\
 On the event $\{u<\tau_{s}(x)\},\varphi_{s,t}(x)=x+\vec{e}(x)W_{s,t},\tau_{t}(\varphi_{s,t}(x))=\tau_{s}(x)<u$  and  
$$\varphi_{t,u}\circ\varphi_{s,t}(x)=x+\vec{e}(x)(W_{s,t}+W_{t,u})=x+\vec{e}(x)W_{s,u}=\varphi_{s,u}(x).$$
On the event $\{\tau_{s}(x)\in]t,u]\}$, we still have $\varphi_{s,t}(x)=x+\vec{e}(x)W_{s,t}$ and $\tau_{t}(\varphi_{s,t}(x))=\tau_{s}(x)\leq u$,
thus $$\varphi_{t,u}\circ\varphi_{s,t}(x)=\vec\varepsilon_{t,u}W^{+}_{t,u}=\vec\varepsilon_{s,u}W^{+}_{s,u}=\varphi_{s,u}(x).$$
since on the event $\{\tau_{s}(x)\in]t,u]\}, \textrm{min}_{s,u}=\textrm{min}_{t,u}$ and $W^{+}_{s,u}=W_{u}-\textrm{min}_{s,u}=W^{+}_{t,u}.$

On the event $\{\tau_{s}(x)\leq t\}\cap \{\tau_{t}(\varphi_{s,t}(x))\leq u\}, \varphi_{s,t}(x)=\vec\varepsilon_{s,t}W^{+}_{s,t}$ and
$$\varphi_{t,u}\circ\varphi_{s,t}(x)=\varphi_{t,u}(\vec\varepsilon_{s,t}W^{+}_{s,t})=\vec\varepsilon_{t,u}W^{+}_{t,u}=\vec\varepsilon_{s,u}W^{+}_{s,u}=\varphi_{s,u}(x)$$
since $W^{+}_{s,\tau_{t}(\varphi_{s,t}(x))}=0$ and thus $\textrm{min}_{s,u}=\textrm{min}_{t,u}$ which implies $\vec\varepsilon_{s,u}=\vec\varepsilon_{t,u}$ and
$W^{+}_{s,u}=W^{+}_{t,u}.$

On the event $\{\tau_{s}(x)\leq t\}\cap \{\tau_{t}(\varphi_{s,t}(x))> u\}, \varphi_{s,t}(x)=\vec\varepsilon_{s,t}W^{+}_{s,t}$ and
$$\varphi_{t,u}\circ\varphi_{s,t}(x)=\varphi_{t,u}(\vec\varepsilon_{s,t}W^{+}_{s,t})=\vec\varepsilon_{s,t}(W^{+}_{s,t}+W_{t,u})=\vec\varepsilon_{s,u}W^{+}_{s,u}=\varphi_{s,u}(x).$$
since in this case  $\textrm{min}_{s,u}=\textrm{min}_{s,t}$ which implies $\vec\varepsilon_{s,u}=\vec\varepsilon_{s,t}$ and
$$\begin{array}{ll}
W^{+}_{s,u}&=W_{u}-\textrm{min}_{s,u}\\
&=W_{u}-W_{s}+W_{s}-\textrm{min}_{s,t}\\
&=W^{+}_{s,t}+W_{t,u}.
\end{array}$$
Thus we have, a.s. $\varphi_{s,u}(x)=\varphi_{t,u}\circ\varphi_{s,t}(x)$ which proves the cocyle property for $\varphi$. It is now easy to check that $\varphi$ is a stochastic flow of mappings in the sense of Definition 4 \cite{MR5010101}. 

Note that $(\varphi_{0,t},t\geq 0)$ is an $W(\alpha_1,\cdots,\alpha_N)$ started at $0$ and therefore satisfies Freidlin-Sheu formula (Theorem 3 \cite{MR5010101}). Let $f\in D(\alpha_1,\cdots,\alpha_N)$, then for all $t\geq0$,
$$f(\varphi_{0,t})=f(0)+\int_0^tf'(\varphi_{0,u})dB_u+\frac{1}{2}\int_0^t f''(\varphi_{0,u})du\ \ a.s.$$
where $B_t=|\varphi_{0,t}|-\tilde L_t(|\varphi_{0,.}|)$ and $\tilde L_t(|\varphi_{0,.}|)$ is the symmetric local time at $0$ of $|\varphi_{0,.}|$. Since $|\varphi_{0,t}|=W_t-\textrm{min}_{0,t}$, we get $B_t=W_t$. Let $x\in D_i\setminus\{0\}$ and $f_i(r)=f(r\vec{e}_i), r\geq0$. Since $\lim_{z\rightarrow0, z\in D_i, z\neq 0}f'(z)$ and $\lim_{z\rightarrow0, z\in D_i, z\neq 0}f''(z)$ exist, we can construct $g$ which is $C^2$ on $\R$ and coincides with $f_i$ on $\R_+$. By It\^o's formula 
$$g(|x|+W_t)=g(|x|)+\int_0^tg'(|x|+W_u)dW_u+\frac{1}{2}\int_0^t g''(|x|+W_u)du\ \ a.s.$$
and so for $t\leqslant \tau_{0}(x)$, we have
$$f(\varphi_{0,t}(x))=f(x)+\int_0^tf'(\varphi_{0,u}(x))dW_u+\frac{1}{2}\int_0^t f''(\varphi_{0,u}(x))du\ \ a.s.$$
Set $\alpha=f(0)+\int_0^{\tau_{0}(x)}f'(\varphi_{0,u})dW_u+\frac{1}{2}\int_0^{\tau_{0}(x)} f''(\varphi_{0,u})du=f(\varphi_{0,\tau_0(x)})=f(0)$ since $W^+_{0,\tau_0(x)}=0$. Then for $t>\tau_{0}(x)$, write
\begin{eqnarray}
 f(\varphi_{0,t}(x))=f(\varphi_{0,t})&=&\alpha+\int_{\tau_{0}(x)}^tf'(\varphi_{0,u})dW_u+\frac{1}{2}\int_{\tau_{0}(x)}^t f''(\varphi_{0,u})du\nonumber\\
&=&f(0)+\int_{\tau_{0}(x)}^tf'(\varphi_{0,u}(x))dW_u+\frac{1}{2}\int_{\tau_{0}(x)}^t f''(\varphi_{0,u}(x))du.\nonumber\
\end{eqnarray}
But $f(x)+\int_0^{\tau_{0}(x)}f'(\varphi_{0,u}(x))dW_u+\frac{1}{2}\int_0^{\tau_{0}(x)}f''(\varphi_{0,u}(x))du=f(\varphi_{0,\tau_0(x)}(x))=f(0)$ and so, for all $t\geq0, f\in D(\alpha_1,\cdots,\alpha_N), x\in G$,
\begin{equation}\label{comes}
f(\varphi_{0,t}(x))=f(x)+\int_0^tf'(\varphi_{0,u}(x))dW_u+\frac{1}{2}\int_0^t f''(\varphi_{0,u}(x))du\ \ a.s.
\end{equation}
Now, let $(\psi,W)$ be a any flow of mappings solution of $(T)$. Lemma 6 \cite{MR5010101} implies
\begin{equation}\label{mak}
 \psi_{0,t}(x)=x+\vec{e}(x)W_{0,t}\ \textrm{for}\ 0\leq t\leq\tau_{0,x}\ \textrm{with}\ \tau_{0,x}\ \textrm{given by (\ref{didi})}.
 \end{equation}
By considering a sequence $(x_k)_{k\geq0}$ converging to $\infty$, this shows that $\sigma(W_t)\subset\sigma(\psi_{0,t}(y), y\in G)$. Therefore, we can define a Wiener stochastic flow $K^{*}$ obtained by filtering $\delta_{\psi}$ with respect to $\sigma(W)$ (Lemma 3-2 (ii) in \cite{MR2060298}) satisfying: $\forall s\leq t, x\in G, K_{s,t}^{*}(x)=E[\delta_{\psi_{s,t}(x)}|\sigma(W)]\  a.s$. In particular $K^{*}$ solves $(T)$ and since $K^{W}$ given by (\ref{loop})
is the unique Wiener solution of $(T)$, we get: $\forall s\leq t, x\in G,\ \ K_{s,t}^{W}(x)=E[\delta_{\psi_{s,t}(x)}|\sigma(W)]\  a.s.$ (see Proposition 8 \cite{MR5010101}). As $K_{0,t}^W(0)$ is supported on $\{W^+_{0,t}\vec{e}_i,\ \ 1\leq i\leq N\}$, we deduce that $|\psi_{0,t}(0)|=W^+_{0,t}$. Combining this with (\ref{mak}), we see that 
$$\inf\{r\geq0: \psi_{0,r}(x)=\psi_{0,r}(0)\}=\tau_{0,x}.$$
This implies $\psi_{0,r}(x)=\psi_{0,r}(0)$ for all $r\geq\tau_{0,x}$ by applying the following
\begin{lemma}
For all $(x_1,\cdots,x_n)\in G^n,$ denote by $\mathbb P_{x_1,\cdots,x_n}$ the law of $(\psi_{0,.}(x_1),\cdots,\psi_{0,.}(x_n))$ in $C(\R_+,G^n)$. Let $T$ be a finite $(\mathcal F_t)$ stopping time where $\mathcal F_t=\sigma(\psi_{0,u}, u\leq t), t\geq0$. Then the law of $(\psi_{0,T+.}(x_1),\cdots,\psi_{0,T+.}(x_n))$ knowing $\mathcal F_T$ is given by $\mathbb P_{\psi_{0,T}(x_1),\cdots,\psi_{0,T}(x_n)}$.
\end{lemma}
Note that $W_{0,.}$ can be recovered out from $W_{0,.}^+$ and consequently $\psi_{0,.}(x)$ is a measurable function of $\psi_{0,.}(0)$ for all $x\in G$. Therefore, for all $(x_1,\cdots,x_n)\in G^n$, $(\psi_{0,\cdot}(x_1),\cdots,\psi_{0,\cdot}(x_n))$ is unique in law since $\psi_{0,\cdot}(0)$ is a Walsh Brownian motion. This completes the proof.
\end{proof}
\subsubsection{The Wiener flow.}
Remark that $K_{s,t}^W(x)=E[\delta_{\varphi_{s,t}(x)}|\sigma(W)]$ which entails that $K^W$ is a stochastic flow of kernels. By conditioning with respect to $\sigma(W)$ in (\ref{comes}), we easily see that $(K^W,W)$ solves $(T)$. In order to finish the proof of Theorem \ref{lamjad} and Corollary \ref{bess}, we need only check the following lemma (the proof of (\ref{thh}) is similar)
\begin{lemma} Under the hypothesis of Proposition \ref{valerie}, we have
$$\displaystyle{\sup_{t\in [s,s+T]}}\beta (K_{s,t}^W(x),K_{s,t}^n(\sqrt{n}x_n))\xrightarrow[{\text{$n\rightarrow +\infty $}}]{} 0\ \ \textrm{a.s.}$$
\end{lemma}

\noindent\begin{proof}
Let $ g:G\longrightarrow \R$ such that $\| g\|_\infty+\displaystyle{\sup_{x \not =y}}\frac {|g(x)-g(y)|}{|x-y|}\leq 1, g(0) = 0$. Then, $$\left|\displaystyle\int_{G}g(y)K_{s,t}^W(x)(dy)-\displaystyle\int_{G}g(y)K_{s,t}^n(\sqrt{n}x_n)(dy)\right|\leq V_{s,t}^{1,n}+V_{s,t}^{2,n}$$
where
$$V_{s,t}^{1,n}=\left|g(x_n+\vec{e}(x_n)S_{s,t}^n){1}_{\{\lfloor nt \rfloor\leq T_{\lfloor ns \rfloor,\sqrt n x_n}\}}-g(x+\vec{e}(x)W_{s,t}){1}_{\{t\leq\tau_{s,x}\}}\right|,$$
$$V_{s,t}^{2,n}=\displaystyle\sum_{j=1}^{N}\alpha_j\left|g(\vec{e}_jW_{s,t}^{+}) 1_{\{t> \tau_{s,x}\}}-g(\vec{e}_jS_{n,s,t}^{+}+o_n){1}_{\{\lfloor nt \rfloor>T_{\lfloor ns \rfloor,\sqrt n x_n}\}}\right|$$
and $o_n\in G$ is a $\sigma(S)$ measurable random variable such that $|o_n|\leq\frac{1}{\sqrt{n}}$, $S_{s,t}^n=S^n_t-S^n_s,\ \ S_{n,s,t}^+ = \frac{1}{\sqrt n} S^+_{\lfloor ns \rfloor,\lfloor nt \rfloor}$. As $\lfloor x\rfloor-1\leq x\leq \lfloor x\rfloor+1$ for all $x\in\R$, we get
$$V_{s,t}^{1,n}\leq \displaystyle{\sup_{t\in I_{n,s,x}}}|x_n+\vec{e}(x_n)S_{s,t}^{(n)}-x-\vec{e}(x)W_{s,t}|+\displaystyle{\sup_{t\in J_{n,s,x}}}|g(x_n+\vec{e}(x_n)S_{s,t}^{(n)})|+\displaystyle{\sup_{t\in K_{n,s,x}}}|g(x+\vec{e}(x)W_{s,t})|$$
\noindent with\ \ \ \ \ \ \ \ \ \ \ \ \ \ \ \ \ \ \ \ \ \ \ \ \ \ \ \ \  $I_{n,s,x}=[s,\tau _{s,x}\vee(\frac{1}{n}+\frac{1}{n} T_{\lfloor ns \rfloor,\sqrt n x_n})],$\\
$$J_{n,s,x}=[\tau _{s,x}, (\frac{1}{n} T_{\lfloor ns \rfloor,\sqrt n x_n}+\frac{1}{n})\vee\tau_{s,x}],\ \  K_{n,s,x}=[\tau _{s,x}\land(\frac{1}{n} T_{\lfloor ns \rfloor,\sqrt n x_n}-\frac{1}{n}),\tau _{s,x}].$$
Using $|g(y)|\leq |y|$, we obtain
$$\displaystyle{\sup_{t\in J_{n,s,x}}}|g(x_n+\vec{e}(x_n)S_{s,t}^{(n)})|+\displaystyle{\sup_{t\in K_{n,s,x}}}|g(x+\vec{e}(x)W_{s,t})|\leq \displaystyle{\sup_{t\in J_{n,s,x}}}||x_n|+S_{s,t}^{(n)}|+\displaystyle{\sup_{t\in K_{n,s,x}}}||x|+W_{s,t}|.$$
Since $\lim\limits_{\substack{n \to +\infty}}\frac{1}{n} T_{\lfloor ns \rfloor,\sqrt n x_n}=\tau_{s,x}$ a.s., the right-hand side converges to $0$. By discussing the cases $x=0, x\neq0$, we easily see that $\lim_{n\rightarrow\infty}\displaystyle{\sup_{t\in I_{n,s,x}}}|x_n+\vec{e} (x_n)S_{s,t}^{(n)}-x-\vec{e}(x)W_{s,t}|=0$ and therefore $\lim_{n\rightarrow\infty}\displaystyle{\sup_{t\in [s,s+T]}}V_{s,t}^{1,n}=0$. By the same manner, we arrive at  $\lim_{n\rightarrow\infty}\displaystyle{\sup_{t\in [s,s+T]}}V_{s,t}^{2,n}=0$ which proves the lemma.
\end{proof}

%
\begin{acknowledgement}
I sincerely thank Yves Le Jan, Olivier Raimond and Sophie Lemaire for very useful discussions. I am also grateful to the referee for his helpful comments.
\end{acknowledgement}
%
%
%

\end{document}